\documentclass[a4paper]{article}

\newcommand{\versiondate}{January 14, 2020}

\usepackage[latin1]{inputenc}
\usepackage{amssymb,amsmath,amsthm}
\usepackage{amsfonts}
\usepackage{enumitem}
\usepackage{ifthen}
\usepackage{url}
\usepackage{hyphenat}

\newcommand{\vect}[1]{\ensuremath{\mathbf{#1}}}
\newcommand{\card}[1]{\ensuremath{\lvert{#1}\rvert}}
\newcommand{\nset}[1]{\ensuremath{[{#1}]}}
\newcommand{\nocc}[2]{\ensuremath{\card{{#2}}_{#1}}}

\newcommand{\preserves}{\ensuremath{\vartriangleright}}

\newcommand{\IN}{\ensuremath{\mathbb{N}}}
\newcommand{\esssorts}[2][S]{\ensuremath{{#1}_{#2}}}
\newcommand{\mcf}[1]{\ensuremath{\langle{#1}\rangle}}
\newcommand{\mcr}[1]{\ensuremath{[{#1}]}}
\newcommand{\lift}[2][k]{\ensuremath{\sharp_{#1}\ifthenelse{\equal{#2}{}}{}{{#2}}}}
\newcommand{\flatten}[2][k]{\ensuremath{\flat_{#1}\ifthenelse{\equal{#2}{}}{}{{#2}}}}

\DeclareMathOperator{\range}{Im}
\DeclareMathOperator{\declaration}{dec}
\DeclareMathOperator{\arity}{ar}
\DeclareMathOperator{\mInv}{mInv}
\DeclareMathOperator{\mPol}{mPol}
\DeclareMathOperator{\Inv}{Inv}
\DeclareMathOperator{\Pol}{Pol}
\DeclareMathOperator{\Mod}{Mod}
\DeclareMathOperator{\mId}{mId}

\DeclareMathOperator{\pr}{pr}
\DeclareMathOperator{\Op}{Op}
\DeclareMathOperator{\Relp}{Relp}
\DeclareMathOperator{\Alg}{Alg}
\DeclareMathOperator{\Relax}{Relax}
\DeclareMathOperator{\sE}{\mathsf{E}}
\DeclareMathOperator{\sR}{\mathsf{R}}
\DeclareMathOperator{\sH}{\mathsf{H}}
\DeclareMathOperator{\sS}{\mathsf{S}}
\DeclareMathOperator{\sP}{\mathsf{P}}
\DeclareMathOperator{\sPfin}{\mathsf{P}_\mathrm{fin}}

\newcommand{\bb}{\mathbf{b}}
\newcommand{\bc}{\mathbf{c}}
\newcommand{\bd}{\mathbf{d}}

\newcommand{\emptyword}{\ensuremath{\varepsilon}}

\theoremstyle{plain}
\newtheorem{theorem}{Theorem}[section]
\newtheorem{proposition}[theorem]{Proposition}
\newtheorem{lemma}[theorem]{Lemma}

\theoremstyle{definition}
\newtheorem{definition}[theorem]{Definition}
\newtheorem{example}[theorem]{Example}

\newtheorem{remark}[theorem]{Remark}

\title{Reflections and powers of multisorted minions}

\author{%
    Erkko Lehtonen\,\textsuperscript{1}, Reinhard P\"oschel\,\textsuperscript{2} \\[1ex]
    \small{\textsuperscript{1, 2} Technische Universit\"at Dresden,
    Institut f\"ur Algebra,} \\
    \small{01062 Dresden,
    Germany} \\[0.75ex]
    \small{\textsuperscript{1} Centro de Matem\'atica e Aplica\c{c}\~oes,
    Faculdade de Ci\^encias e Tecnologia,} \\
    \small{Universidade Nova de Lisboa,
    Quinta da Torre,
    2829-516 Caparica,
    Portugal}}

\date{\versiondate}

\begin{document}
\maketitle

\begin{abstract}
\noindent
Classes of multisorted minions closed under extensions, reflections, and direct powers are considered from a relational point of view. As a generalization of a result of Barto, Opr\v{s}al, and Pinsker, the closure of a multisorted minion is characterized in terms of constructions on multisorted relation pairs which are invariant for minions.
\end{abstract}




\section{Introduction}

In the famous ``wonderland paper'' by \textsc{L. Barto}, \textsc{J. Opr\v{s}al} and \textsc{M. Pinsker} (the first version appeared in 2015), a new algebraic notion saw the light of day: the \emph{reflection} of an algebra (applicable also to a single function) \cite[Definition~4.1]{wonderland}.
This notion was introduced for the primary purpose of investigating the computational complexity of constraint satisfaction problems (CSPs).

From the algebraic point of view, reflections generalize at the same time both subalgebras and homomorphic images; however, they no longer preserve arbitrary identities but only so\hyp{}called h1\hyp{}identities, nowadays also known as \emph{minor identities.}
Furthermore, for the operations of an algebra, instead of clones, a weaker notion had to be considered where composition of functions is no longer required: the so\hyp{}called \emph{minions} (also called minor\hyp{}closed classes or clonoids \cite{AicMay,CouFol,Pippenger,Sparks}).
The well\hyp{}known Galois connection $\Pol$--$\Inv$ (for clones and relational clones, induced by the property of preservation of a relation by a function) does not work well for describing these minions as Galois closures.
It turned out that the Galois theory introduced by \textsc{N. Pippenger} \cite{Pippenger} for minor\hyp{}closed sets of functions provides the right tool for minions: instead of (invariant) relations one has to consider pairs of relations.

Already Pippenger dealt with functions $f \colon A^n \to B$ between two sets (an approach which was later also used for so\hyp{}called promise constraint satisfaction problems, PCSPs \cite{BraGur}).
Therefore it was natural to ask if and how all these notions work for \emph{multisorted algebras} (i.e., algebras with several base sets) in general, in order to provide a systematic algebraic treatment of everything that pertains or might pertain to reflections.
This was done in \cite{LehPoeWal-R1} where Birkhoff\hyp{}like theorems were established for multisorted algebras and minor identities, and in \cite{LehPoeWal-R2} where the Galois connection $\Pol$--$\Inv$ is generalized to the multisorted case, yielding $\mPol$--$\mInv$ with minions as Galois closed sets.

In the present paper we use results from \cite{LehPoeWal-R1} and \cite{LehPoeWal-R2} in order to treat one of the main results of the wonderland paper, namely, the characterization of the ${\sE}{\sR}{\sP}$\hyp{}closure \cite[Theorem~1.3]{wonderland}, also for multisorted algebras.
In particular, we ask how this ${\sE}{\sR}{\sP}$\hyp{}closure can be characterized from the relational point of view (i.e., by constructions on invariant relation pairs of minions).
To this end, we introduce reflections (and coreflections) also for relation pairs (see Section~\ref{reflection-coreflection}), and these new concepts turn out to provide a suitable tool for enlightening reflections from the relational point of view.

The paper is organized as follows.
In Section~\ref{sec:preliminaries},
we review basic concepts and well\hyp{}known results related to multisorted operations and algebras, minors, reflections, and relation pairs that will be needed in the subsequent sections.
In Section~\ref{sec:refl-corefl},
we define reflections and coreflections as well as liftings and flattenings of relation pairs and establish a few auxiliary results.

The main results are given in Section~\ref{sec:results}.
At first, we define the operators $\sE$, $\sR$, $\sP$ (extensions, reflections, direct powers) and for each of them we show how they can be characterized via invariant relation pairs (Propositions \ref{prop:E}, \ref{prop:R}, \ref{prop:P}).
Combining the above, we then characterize the ${\sR}{\sP}$-closure in Theorem~\ref{thm:main-RP}.
We would like to underline that this result is of interest on its own right if one wants to consider only algebras of a fixed type.
Finally, with Theorem \ref{thm:main-ERP} we put forward the multisorted counterpart of the wonderland theorem~\cite[Theorem~1.3]{wonderland}; this also includes the characterization via relational constructions (cf.~Theorem~\ref{thm:main-ERP}\ref{main-ERP:fin}\ref{main-ERP:mcc}).

We would like to mention at this point that usual (one\hyp{}sorted) algebras as well as clones are special multisorted algebras and special minions.
Therefore all results of this paper can be (re)interpreted and used for the classical case, too.
It is possible to characterize clones $C$ as minions with special properties for its invariant relation pairs $\mInv C$ (and $\Inv C$ can be defined from $\mInv C$).
This is, however, not in the main focus of the current paper and therefore will not be considered and proved here.


\section{Preliminaries}
\label{sec:preliminaries}
\subsection*{Multisorted operations}

We start by briefly recalling basic concepts in the theory of multisorted sets and multisorted operations.
We follow the presentation of \cite{LehPoeWal-R1,LehPoeWal-R2}, which is largely based on the terminology used in the book by Wechler~\cite{Wechler}, and we refer the reader to those publications for precise definitions, further details, and explanations not given here.

We denote by $\IN$ and $\IN_{+}$ the set of nonnegative integers and the set of positive integers, respectively.
For $n \in \IN$, let $\nset{n} := \{1, \dots, n\}$.

We write $W(S)$ for the set of all words over a set $S$ is denoted by $W(S)$; this includes the empty word $\emptyword$.
We denote by $\card{w}$ the length of a word $w \in W(S)$ and by $\nocc{s}{w}$ the number of occurrences of a letter $s \in S$ in $w$.

Let $S$ be a set of elements called \emph{sorts.}
An $S$\hyp{}indexed family of sets is called an \emph{$S$\hyp{}sorted set.}
Let $A = (A_s)_{s \in S}$ and $B = (B_s)_{s \in S}$ be $S$\hyp{}sorted sets.
We say that $A$ is an \emph{\textup{(}$S$\hyp{}sorted\textup{)} subset} of $B$ and we write $A \subseteq B$ if $A_s \subseteq B_s$ for all $s \in S$.
The \emph{union,} \emph{intersection,} and \emph{direct product} of $S$\hyp{}sorted sets $A$ and $B$, as well as \emph{powers} of $A$ are defined componentwise: $A \cup B := (A_s \cup B_s)_{s \in S}$, $A \cap B := (A_s \cap B_s)_{s \in S}$, and $A \times B := (A_s \times B_s)_{s \in S}$, $A^k := (A_s^k)_{s \in S}$ for any $k \in \IN$.
Let $\esssorts{A} := \{s \in S \mid A_s \neq \emptyset\}$ be the set of \emph{essential sorts} of $A$.

For a word $w = s_1 s_2 \dots s_n \in W(S)$, let $A_w := A_{s_1} \times A_{s_2} \times \dots \times A_{s_n}$.
Note that $A_\emptyword = \{\emptyset\}$.
A pair $(w,s) \in W(S) \times S$ is called a \emph{declaration} over $S$.
A declaration $(w,s)$ with $w = s_1 \dots s_n$ is \emph{reasonable} in $A$ if $A_s = \emptyset$ implies $A_{s_i} = \emptyset$ for some $i$, or, equivalently, if $A_w \neq \emptyset$ implies $A_s \neq \emptyset$.
Note that the declaration $(\emptyword,s)$ is reasonable in $A$ if and only if $A_s \neq \emptyset$.

Any function of the form $f \colon A_w \to A_s$ is called an \emph{$S$\hyp{}sorted operation} on $A$; the pair $(w,s)$ is necessarily a reasonable declaration in $A$.
The pair $(w,s)$ is called the \emph{declaration} of $f$, the word $w$ is called the \emph{arity} of $f$, and the element $s$ is called the \emph{\textup{(}output\textup{)} sort} of $f$.
The elements of $S$ occurring in the word $w$ are called the \emph{input sorts} of $f$.
We denote the declaration and arity of
$f$ by $\declaration(f)$ and $\arity(f)$, respectively.

We denote the set of all $S$\hyp{}sorted operations of declaration $(w,s)$ on $A$ by $\Op^{(w,s)}(A)$.
Let $\Op(A)$ be the set of all $S$\hyp{}sorted operations on $A$, i.e.,
\[
\Op(A) := \bigcup \{ \Op^{(w,s)}(A) \mid (w,s) \in W(S) \times S \}.
\]

An \emph{$S$\hyp{}sorted mapping} $f$ from $A$ to $B$, denoted by $f \colon A \to B$, is a family $(f_s)_{s \in S}$ of maps $f_s \colon A_s \to B_s$.

\subsection*{Minors, reflections, powers}

\begin{definition}\label{def:minor}
Given an $S$\hyp{}sorted operation $f \colon A_w \to A_s$ with $w = s_1 \dots s_n$, a word $u = u_1 \dots u_m \in W(S)$ such that $\{s_1, \dots, s_n\} \subseteq \{u_1, \dots, u_m\}$ and a map $\sigma \colon \nset{n} \to \nset{m}$ satisfying $s_i = u_{\sigma(i)}$ for all $i \in \nset{n}$ (sort compatibility), define the function $f^u_\sigma \colon A_u \to A_s$ of declaration $(u,s)$ by the rule $f^u_\sigma(\vect{a}) := f(\vect{a} \sigma)$, for all $\vect{a} \in A_u$, i.e.,
\[
f^u_\sigma(a_1, \dots, a_m) :=
f(a_{\sigma(1)}, \dots, a_{\sigma(n)})
\]
for all $a_i \in A_{u_i}$ ($1 \leq i \leq m$).
Such a function $f^u_\sigma$ is called a \emph{minor} of $f$.
A set $F \subseteq \Op(A)$ is \emph{minor\hyp{}closed} or a \emph{minion} if it contains all minors of its members.
For $F \subseteq \Op(A)$, we denote by $\mcf{F}$ the minion generated by $F$, i.e., the smallest minion containing $F$.
\end{definition}

\begin{definition}
\label{def:reflection}
Let $A$ and $B$ be $S$\hyp{}sorted sets.
A \emph{reflection} of $A$ into $B$ is a pair $(h,h')$ of $\esssorts{B}$\hyp{}sorted mappings $h = (h_s)_{s \in \esssorts{B}}$, $h' = (h'_s)_{s \in \esssorts{B}}$, $h_s \colon B_s \to A_s$, $h'_s \colon A_s \to B_s$.
Reflections of $A$ into $B$ exist if and only if $\esssorts{B} \subseteq \esssorts{A}$.

Assume that $\esssorts{B} \subseteq \esssorts{A}$ and $(h,h')$ is a reflection of $A$ into $B$.
If $(w,s) \in W(S) \times S$ is a declaration that is reasonable in both $A$ and $B$ and $f \colon A_w \to A_s$, then we can define the \emph{$(h,h')$\hyp{}reflection} of $f$ as the map $f_{(h,h')} \colon B_w \to B_s$ that is the empty map if $B_w = \emptyset$ and is otherwise given by the rule
\[
f_{(h,h')}(b_1, \dots, b_n) := h'_s ( f ( h_{s_1} ( b_1 ), \dots, h_{s_n} ( b_n ) ) )
\]
for all $(b_1, \dots, b_n) \in B_w$,\hspace{-0.6pt} which we may write briefly as $f_{(h,h')}(\vect{b}) = h'_s(f(h_w(\vect{b})))$ for $\vect{b} = (b_1, \dots, b_n) \in B_w$.
We say that an $S$\hyp{}sorted operation $g$ is a \emph{reflection} of $f$ if $g$ is an $(h,h')$\hyp{}reflection of $f$ for some reflection $(h,h')$.

Let $F \subseteq \Op(A)$.
If $\declaration(f)$ is reasonable in $B$ for all $f \in F$, then the \emph{$(h,h')$\hyp{}reflection} of $F$ is defined as $F_{(h,h')} := \{f_{(h,h')} \mid f \in F\}$.
Sets of operations of the form $F_{(h,h')}$ for some reflection $(h,h')$ are called \emph{reflections} of $F$.
\end{definition}

\begin{definition}
Let $f \in \Op(A)$ and $I$ an arbitrary index set.
Define $f^{\otimes I} \in \Op(A^I)$ by componentwise application of $f$ to indexed families, i.e.,
for all $(a^j_i)_{i \in I} \in A^I$ ($1 \leq j \leq n$),
\[
f^{\otimes I}((a^1_i)_{i \in I}, \dots, (a^n_i)_{i \in I})
:=
(f(a^1_i, a^2_i, \dots, a^n_i))_{i \in I}.
\]
We refer to $f^{\otimes I}$ as a \emph{direct power} of $f$.
For $F \subseteq \Op(A)$, let $F^{\otimes I} := \{f^{\otimes I} \mid f \in F\}$.
If the index set $I$ is finite, then $f^{\otimes I}$ is called a \emph{finite direct power} of $f$.
Whenever $I = \nset{k}$ for some $k \in \IN$,
we may write simply $f^{\otimes k}$ and $F^{\otimes k}$ for $f^{\otimes \nset{k}}$ and $F^{\otimes \nset{k}}$.
\end{definition}

The following two propositions show that reflections and direct powers of minions are minions.

\begin{proposition}[{\cite[Proposition~5.2]{LehPoeWal-R2}}]
\label{prop:refl-mc}
Let $A$ and $B$ be $S$\hyp{}sorted sets.
Let $F \subseteq \Op(A)$, and let $(h,h')$ be a reflection of $A$ into $B$ such that $F_{(h,h')}$ is defined.
If $F$ is a minion, then $F_{(h,h')}$ is a minion.
\end{proposition}

\begin{proposition}
\label{prop:power-mc}
Let $A$ be an $S$\hyp{}sorted set, and let $F \subseteq \Op(A)$ be a minion.
Then for any $k \in \IN$, $F^{\otimes k}$ is a minion.
\end{proposition}

\begin{proof}
Let $g \in F^{\otimes k}$ with $\declaration(g) = (w,s)$, $w = s_1 \dots s_n$.
Then $g = f^{\otimes k}$ for some $f \in F$ with $\declaration(f) = (w,s)$.
It is straightforward to verify that for all $u = u_1 \dots u_m \in W(S)$ and $\sigma \colon \nset{n} \to \nset{m}$ such that $f_\sigma^u$ is defined,
it holds that $(f^{\otimes k})^u_\sigma = (f^u_\sigma)^{\otimes k}$.
Thus $F^{\otimes k}$ is minor\hyp{}closed if $F$ is minor\hyp{}closed.
\end{proof}

\subsection*{Multisorted algebras and identities}

A (\emph{multisorted similarity}) \emph{type} is a triple $\tau = (S, \Sigma, \declaration)$, where $S$ is a set of sorts, $\Sigma$ is a set of \emph{function symbols,} and $\declaration \colon \Sigma \to W(S) \times S$ is a mapping.
If $f \in \Sigma$ and $\declaration(f) = (w,s)$, we say that $f$ has \emph{arity} $w$ and \emph{sort} $s$.
A (\emph{multisorted}) \emph{algebra of type $\tau$} is a system $\mathbf{A} = (A; \Sigma^\mathbf{A})$, where $A$ is an $S$\hyp{}sorted set, called the \emph{carrier} (or the \emph{universe}) of $\mathbf{A}$, and $\Sigma^\mathbf{A} = (f^\mathbf{A})_{f \in \Sigma}$ is a family of $S$\hyp{}sorted operations on $A$, each $f^\mathbf{A}$ of declaration $\declaration(f)$.
Denote by $\Alg(\tau)$ the class of all multisorted algebras of type $\tau$.

Homomorphisms, subalgebras, and direct products of multisorted algebras are defined in the expected way.
Let $\mathbf{A} = (A, \Sigma^\mathbf{A})$ and $\mathbf{B} = (B, \Sigma^\mathbf{B})$ be multisorted algebras of type $\tau = (S, \Sigma, \declaration)$.
For a reflection $(h,h')$ of $A$ into $B$, the algebra $\mathbf{B}$ is called the \emph{$(h,h')$\hyp{}reflection} of $\mathbf{A}$ if $f^\mathbf{B} = (f^\mathbf{A})_{(h,h')}$ for all $f \in \Sigma$.
The algebra $\mathbf{B}$ is a \emph{reflection} of $\mathbf{A}$ if $\mathbf{B}$ is an $(h,h')$\hyp{}reflection of $\mathbf{A}$ for some reflection $(h,h')$ of $A$ into $B$.

Let $\mathcal{K}$ be a class of multisorted algebras of a fixed type.
Denote by $\sS \mathcal{K}$, $\sH \mathcal{K}$, $\sP \mathcal{K}$, and $\sR \mathcal{K}$ the class of all subalgebras, homomorphic images, direct products, and reflections of members of $\mathcal{K}$, respectively.

Terms can be defined in the setting of multisorted algebras much in the same way as in the classical one\hyp{}sorted case.
One has to be a bit more careful in defining identities.
In contrast to the one\hyp{}sorted case, it is not sufficient to define an identity to be a pair of terms; one also has to specify the variables that are to be valuated when one decides whether an identity holds in an algebra.
With this in mind, the definition of a multisorted algebra satisfying an identity can be laid out in the expected way.
For further details, see \cite[Section~3]{LehPoeWal-R1}.

We consider terms of type $\tau = (S, \Sigma, \declaration)$ over the \emph{standard set of variables} $X = (X_s)_{s \in S}$ with $X_s := \{x^s_i \mid i \in \IN_{+}\}$.
A term is called a \emph{minor term} if it contains exactly one occurrence of a function symbol.
Thus a minor term is of the form $f \sigma(1) \dots \sigma(n)$ where $f \in \Sigma$ with $\declaration(f) = (w, s)$, $w = s_1 \dots s_n$, and $\sigma \colon \nset{n} \to X$ is a map satisfying $\sigma(i) \in X_{s_i}$ for all $i \in \nset{n}$; this term will be denoted by $f_\sigma$.
A word $u = u_1 \dots u_m \in W(S)$ is a \emph{feasible arity for $f_\sigma$} if for every $i \in \nset{n}$ it holds that if $\sigma(i) = x^{s_i}_j$ then $\nocc{s_i}{u} \geq j$.
If $u$ is a feasible arity for $f_\sigma$, then we can define the \emph{term operation} of arity $u$ induced by $f_\sigma$ on $\mathbf{A}$, denoted by $(f_\sigma^u)^\mathbf{A}$ and defined by $(f_\sigma^u)^\mathbf{A} \colon A_u \to A_s$,
$(f_\sigma^u)^\mathbf{A}(a_1, \dots, a_m) := f^\mathbf{A}(a_{\nu(1)}, \dots, a_{\nu(n)})$,
where $\nu(i) = \ell$ if and only if $\sigma(i) = x^{s_i}_j$ and $\ell$ is the position of the $j$\hyp{}th occurrence of $s_i$ in $u$.

\begin{example}
In order to illustrate the notions introduced above,
consider the multisorted similarity type $\tau = (S, \Sigma, {\declaration})$, where $S = \{1,2,3\}$, $\Sigma = \{f\}$, $\declaration(f) = (12321,2)$.
Then $f x^1_3 x^2_2 x^3_4 x^2_2 x^1_2$ is a minor term of type $\tau$ over the standard set $X$ of variables.
We can write this term as $f_\sigma$ with $\sigma \colon \nset{5} \to X$, $1 \mapsto x^1_3$, $2 \mapsto x^2_2$, $3 \mapsto x^3_4$, $4 \mapsto x^2_2$, $5 \mapsto x^1_2$.
The word $u = 333221333221333221 \in W(S)$ is a feasible arity for $f_\sigma$ (because $u$ has at least 3 occurrences of $1$, at least 2 occurrences of $2$, at least 4 occurrences of $3$, and so on).
Given an algebra $\mathbf{A} = (A; \Sigma^\mathbf{A})$ of type $\tau$ (that is, $\Sigma^\mathbf{A} = \{f^\mathbf{A}\}$ and $f^\mathbf{A}$ is an operation of declaration $(12321,2)$ on $A$), the term $f_\sigma$ induces the term operation $(f_\sigma^u)^\mathbf{A}$ of arity $u$ on $\mathbf{A}$, defined by the rule
$(f_\sigma^u)^\mathbf{A}(a_1, a_2, \dots, a_{18}) := f^\mathbf{A}(a_{18}, a_5, a_7, a_5, a_{12})$
(because the 3rd occurrence of $1$ lies at the 18th position in $u$, the 2nd occurrence of $2$ at the 5th position, and so on).
\end{example}

An \emph{identity} is a triple $(S', t_1, t_2)$, usually written as $t_1 \approx_{S'} t_2$, where $t_1$ and $t_2$ are terms of type $\tau$ over $X$ and $S' \subseteq S$ such that the sorts of the variables occurring in the two terms belong to the set $S'$.
An algebra $\mathbf{A}$ \emph{satisfies} the identity $t_1 \approx_{S'} t_2$ if the terms $t_1$ and $t_2$ take the same value under every valuation of variables from $X_s$, $s \in S'$, in $A$.
A \emph{minor identity} is an identity where the two terms are minor terms.
An algebra $\mathbf{A}$ satisfies a minor identity $f_\sigma \approx_{S'} g_\pi$ if and only if for any (equivalently, for one) $u = u_1 \dots u_m \in W(S)$ with $\{u_1, \dots, u_m\} = S'$ that is a feasible arity for both $f_\sigma$ and $g_\pi$ it holds that
$(f^\mathbf{A})_\sigma^u = (g^\mathbf{A})_\pi^u$.

The satisfaction relation induces a Galois connection between multisorted algebras and minor identities.
For a class $\mathcal{K}$ of algebras, let $\mId \mathcal{K}$ be the set of all minor identities satisfied by every algebra in $\mathcal{K}$, and for a class $\mathcal{J}$ of minor identities, let $\Mod \mathcal{J}$ be the class of all algebras satisfying every identity in $\mathcal{J}$.
This Galois connection was investigated in \cite{LehPoeWal-R1}, where it was shown that the Galois closed classes of multisorted algebras (minor\hyp{}equational classes) are precisely the reflection\hyp{}closed varieties.

\begin{theorem}[{\cite[Theorem~5.2]{LehPoeWal-R1}}]
\label{thm:Mod-mId-K}
Let $\mathcal{K}$ be a class of multisorted algebras of a fixed type.
Then $\Mod \mId \mathcal{K} = \mathop{{\sR}{\sP}} \mathcal{K}$.
\end{theorem}

\subsection*{Relation pairs}

Let $A$ be an $S$\hyp{}sorted set and $m \in \IN$.
An $m$\hyp{}ary \emph{$S$\hyp{}sorted relation} on $A$ is a family $R = (R_s)_{s \in S}$ where $R_s \subseteq A_s^m$ for every $s \in S$.
(Note that for $m > 0$, the only $m$\hyp{}ary relation on the empty set is $\emptyset$, and that there are precisely two $0$\hyp{}ary relations on any set: $\emptyset$ and $\{\emptyset\}$.)
An $m$\hyp{}ary \emph{$S$\hyp{}sorted relation pair} of $A$ is a pair $(R,R')$ where $R$ and $R'$ are $m$\hyp{}ary $S$\hyp{}sorted relations on $A$.
Denote by $\Relp^{(m)}$ the set of all $m$\hyp{}ary $S$\hyp{}sorted relation pairs on $A$ and by $\Relp(A)$ the set of all $S$\hyp{}sorted relation pairs on $A$, i.e.,
\[
\Relp(A) := \bigcup_{m \in \IN} \Relp^{(m)}(A).
\]

Let $f \colon A_w \to A_s$ with $w := s_1 \dots s_n$, and let $(R,R')$ be an $m$\hyp{}ary $S$\hyp{}sorted relation pair on $A$.
The operation $f$ \emph{preserves} the relation pair $(R,R')$ (or $f$ is a \emph{polymorphism} of $(R,R')$, or $(R,R')$ is an \emph{invariant relation pair} of $f$), denoted by $f \preserves (R,R')$, if for all $(a_{1i}, a_{2i}, \dots, a_{mi}) \in R_{s_i}$ ($i = 1, \dots, n$), it holds that
\[
( f(a_{11}, a_{12}, \dots, a_{1n}), f(a_{21}, a_{22}, \dots, a_{2n}), \dots, f(a_{m1}, a_{m2}, \dots, a_{mn}) ) \in R'_s.
\]

The preservation relation $\preserves$ induces a Galois connection between $S$\hyp{}sorted operations and $S$\hyp{}sorted relation pairs on $A$, consisting of the maps $\mPol \mathord{:} \linebreak \mathcal{P}(\Relp(A)) \to \mathcal{P}(\Op(A))$ and $\mInv \colon \mathcal{P}(\Op(A)) \to \mathcal{P}(\Relp(A))$ given by
\begin{align*}
\mPol(Q) &:= \{ f \in \Op(A) \mid \forall (R,R') \in Q \colon f \preserves (R,R') \}, \\
\mInv(F) &:= \{ (R,R') \in \Relp(A) \mid \forall f \in F \colon f \preserves (R,R') \},
\end{align*}
for any $F \subseteq \Op(A)$ and $Q \subseteq \Relp(A)$.

The closed sets of operations and relation pairs with respect to this Galois connection were described in \cite{LehPoeWal-R2} as the minor\hyp{}closed classes of operations (i.e., minions) and the so\hyp{}called minor\hyp{}closed classes of relation pairs, respectively.

Analogues of the ``elementary operations'' $\zeta$, $\tau$, $\pr$, $\times$, and $\wedge$ on relations (see \cite[Section~II.2.3]{Lau}, \cite[Subsections 1.1.7 and 1.1.9]{PosKal}) can be defined for $S$\hyp{}sorted relation pairs (see \cite[pp.\ 70--71]{LehPoeWal-R2}) by applying each operation componentwise and in parallel in each sort.
A relation pair $(R,R')$ is a \emph{relaxation} of $(\tilde{R}, \tilde{R}')$ if $R \subseteq \tilde{R}$ and $R' \supseteq \tilde{R}'$.
For an arbitrary equivalence relation $\varrho$ on $\nset{m}$, let $\delta^m_\varrho := (\delta^m_{\varrho,s})_{s \in S}$, where
\[
\delta^m_{\varrho,s} := \{ (a_1, \dots, a_m) \in A^m_s \mid (i,j) \in \varrho \implies a_i = a_j \}.
\]
Relation pairs of the form $(\delta^m_\varrho, \delta^m_\varrho)$ are called \emph{diagonal relation pairs.}
A set $Q \subseteq \Relp(A)$ of relation pairs is \emph{minor\hyp{}closed} if it contains the diagonal relation pairs and is closed under the elementary operations $\zeta$, $\tau$, $\pr$, $\times$, $\wedge$, relaxations, and arbitrary intersections.
For $Q \subseteq \Relp(A)$, we denote by $\mcr{Q}$ the minor\hyp{}closure of $Q$, i.e., the smallest minor\hyp{}closed set of relation pairs on $A$ that contains $Q$.

\begin{theorem}[{\cite[Theorems 4.10, 4.16]{LehPoeWal-R2}}]
\label{thm:mPolmInv-mInvmPol}
Let $A := (A_s)_{s \in S}$ be an $S$\hyp{}sorted set, and assume that the sets $A_s$ are all finite.
\begin{enumerate}[label={\upshape(\roman*)}]
\item\label{mPolmInv} Let $F \subseteq \Op(A)$.
Then $F = \mPol Q$ for some $Q \subseteq \Relp(A)$ if and only if $F$ is a minion.
Consequently, $\mcf{F} = \mPol \mInv F$ for any $F \subseteq \Op(A)$.

\item\label{mInvmPol} Let $Q \subseteq \Relp(A)$.
Then $Q = \mInv F$ for some $F \subseteq \Op(A)$ if and only if $Q$ is minor\hyp{}closed.
Consequently, $\mcr{Q} = \mInv \mPol Q$ for any $Q \subseteq \Relp(A)$.
\end{enumerate}
\end{theorem}


\section{Reflection, coreflection, lifting, and flattening of relation pairs}\label{sec:refl-corefl}
\label{reflection-coreflection}

\subsection*{Reflection and coreflection}

In this subsection, we introduce a pair of new concepts that forms a counterpart to reflections of operations: reflections and coreflections of relation pairs.
These prove to be useful for describing reflections of minions in terms of invariant relation pairs.

\begin{definition}
Let $A$ and $B$ be $S$\hyp{}sorted sets, and let $(h,h')$ be a reflection of $A$ into $B$.
Let $(R,R')$ be an $S$\hyp{}sorted relation pair on $A$.
The \emph{$(h,h')$\hyp{}reflection} of $(R,R')$ is the $S$\hyp{}sorted relation pair $(R,R')_{(h,h')}$ on $B$ given by
$(R,R')_{(h,h')} := (h^{-1}(R), h'(R'))$,
where
$h^{-1}(R) := (T_s)_{s \in S}$ and $h'(R') := (T'_s)_{s \in S}$ with
\[
T_s :=
\begin{cases}
h_s^{-1}(R_s), & \text{if $B_s \neq \emptyset$,} \\
\emptyset, & \text{if $B_s = \emptyset$,}
\end{cases}
\qquad\quad
T'_s :=
\begin{cases}
h'_s(R'_s), & \text{if $B_s \neq \emptyset$,} \\
\emptyset, & \text{if $B_s = \emptyset$,}
\end{cases}
\]
and
\begin{align*}
h_s^{-1}(R_s) &:= \{(a_1, \dots, a_m) \in B_s^m \mid (h_s(a_1), \dots, h_s(a_m)) \in R_s\}, \\
h'_s(R'_s) &:= \{(h'_s(a_1), \dots, h'_s(a_m)) \mid (a_1, \dots, a_m) \in R'_s\}.
\end{align*}
Let $(T,T')$ be an $S$\hyp{}sorted relation pair on $B$.
The \emph{$(h,h')$\hyp{}coreflection} of \linebreak $(T,T')$ is the $S$\hyp{}sorted relation pair $(T,T')^{(h,h')}$ on $A$ given by
$(T,T')^{(h,h')} := \linebreak (h(T), h'^{-1}(T'))$,
where $h(T) := (R_s)_{s \in S}$ and $h'^{-1}(T') := (R'_s)_{s \in S}$ with
\[
R_s :=
\begin{cases}
h_s(T_s), & \text{if $B_s \neq \emptyset$,} \\
\emptyset, & \text{if $B_s = \emptyset$,}
\end{cases}
\qquad\quad
R'_s :=
\begin{cases}
{h'_s}^{-1}(T'_s), & \text{if $B_s \neq \emptyset$,} \\
\emptyset, & \text{if $B_s = \emptyset$,}
\end{cases}
\]

Let $Q \subseteq \Relp(A)$ be a set of relation pairs on $A$.
The \emph{$(h,h')$\hyp{}reflection} of $Q$ is the set $Q_{(h,h')} := \{(R,R')_{(h,h')} \mid (R,R') \in Q\}$ of relation pairs on $B$.
We define \emph{$(h,h')$\hyp{}coreflections} of sets of relation pairs analogously.
A relation pair (a set of relation pairs) is called a \emph{reflection} (a \emph{coreflection}) of another relation pair (set of relation pairs) if the former is the $(h,h')$\hyp{}reflection ($(h,h')$\hyp{}coreflection) of the latter for some reflection $(h,h')$.
\end{definition}

\begin{proposition}[{\cite[Proposition~5.4]{LehPoeWal-R2}}]
\label{prop:refl-mInv}
Let $A$ and $B$ be $S$\hyp{}sorted sets, $(R,R') \in \Relp(A)$, $(T,T') \in \Relp(B)$, and let $(h,h')$ be a reflection of $A$ into $B$.
Let $f \in \Op(A)$, and assume that $\declaration(f)$ is reasonable in $B$.
Then the following statements hold.
\begin{enumerate}[label={\upshape(\roman*)}]
\item\label{f-fhh} If $f \preserves (R,R')$ then $f_{(h,h')} \preserves (R,R')_{(h,h')}$.
\item\label{fhh-f} If $f_{(h,h')} \preserves (T,T')$ then $f \preserves (T,T')^{(h,h')}$.
\item\label{mInvFhh} If $F \subseteq \Op(A)$ and $\declaration(f)$ is reasonable in $B$ for all $f \in F$, then
\[
\mInv F_{(h,h')} = \{ (T,T') \in \Relp(B) \mid (T,T')^{(h,h')} \in \mInv F \}.
\]
\end{enumerate}
\end{proposition}

The last statement can be written in alternative form.
For a set $Q \subseteq \Relp(A)$ of relation pairs, denote by $\Relax(Q)$ the set of all relaxations of members of $Q$.

\begin{proposition}
\label{prop:Relax-mInv}
Let $A$ and $B$ be $S$\hyp{}sorted sets, let $F \subseteq \Op(A)$ and assume that $\declaration(f)$ is reasonable in $B$ for all $f \in F$, and let $(h,h')$ be a reflection of $A$ into $B$.
Then $\mInv F_{(h,h')} = \Relax((\mInv F)_{(h,h')})$.
\end{proposition}

\begin{proof}
Let $(T,T') \in \mInv F_{(h,h')}$.
By Proposition~\ref{prop:refl-mInv}\ref{mInvFhh}, $(h(T),h'^{-1}(T')) \in \mInv F$,
so $(h^{-1}(h(T)), h'(h'^{-1}(T'))) \in (\mInv F)_{(h,h')}$.
Since $T \subseteq h^{-1}(h(T))$ and $T' \supseteq h'(h'^{-1}(T'))$,
$(T,T')$ is a relaxation of $(h^{-1}(h(T)), h'(h'^{-1}(T')))$; hence $(T,T') \in \Relax((\mInv F)_{(h,h')})$.

For the converse inclusion, let $(T,T') \in \Relax((\mInv F)_{(h,h')})$.
Then there exists $(\tilde{T},\tilde{T}') \in (\mInv F)_{(h,h')}$ such that $T \subseteq \tilde{T}$, $T' \supseteq \tilde{T}'$,
and $(\tilde{T},\tilde{T}') = (R,R')_{(h,h')} = (h^{-1}(R), h'(R))$ for some $(R,R') \in \mInv F$.
Thus $f \preserves (R,R')$ for all $f \in F$.
By Proposition~\ref{prop:refl-mInv}\ref{f-fhh}, $f_{(h,h')} \preserves (h^{-1}(R),h'(R))$ for every $f \in F$, so $(\tilde{T},\tilde{T}') \in \mInv F_{(h,h')}$.
By Theorem~\ref{thm:mPolmInv-mInvmPol}\ref{mInvmPol}, $\mInv F_{(h,h')}$ is minor\hyp{}closed and hence contains all relaxations of its members; therefore $(T,T') \in \mInv F_{(h,h')}$.
\end{proof}

\begin{lemma}
\label{lem:mPolRelaxQ}
For any $Q \subseteq \Relp(A)$, we have $\mPol Q = \mPol \Relax Q$.
\end{lemma}

\begin{proof}
Since $Q \subseteq \Relax Q$, we have $\mPol \Relax Q \subseteq \mPol Q$ by the basic properties of Galois connections.
In order to prove the converse inclusion $\mPol Q \subseteq \mPol \Relax Q$, let $f \in \mPol Q$, and let $(R,R') \in \Relax Q$.
Then $(R,R')$ is a relaxation of some $(T,T') \in Q$.
Since $R \subseteq T$ and $T' \subseteq R'$ and $f \preserves (T,T')$, it is clear that $f \preserves (R,R')$.
Therefore $f \in \mPol \Relax Q$.
\end{proof}

\subsection*{Lifting and flattening}

We now define another pair of useful concepts, two maps that provide translations between relation pairs defined on a multisorted set $A$ and ones defined on a finite power of $A$: lifting and flattening.

\begin{definition}
The sets $A^{nk}$ and $(A^k)^n$ are obviously in a one\hyp{}to\hyp{}one correspondence via
the \emph{lifting} map
$\lift{} \colon A^{nk} \to (A^k)^n$
and its inverse, the \emph{flattening} map
$\flatten{} \colon (A^k)^n \to A^{nk}$,
defined by
\begin{gather*}
\lift{(a_{11}, \dots, a_{1k}, \dots, a_{n1}, \dots, a_{nk})} := ((a_{11}, \dots, a_{1k}), \dots, (a_{n1}, \dots, a_{nk})), \\
\flatten{((a_{11}, \dots, a_{1k}), \dots, (a_{n1}, \dots, a_{nk}))} := (a_{11}, \dots, a_{1k}, \dots, a_{n1}, \dots, a_{nk}).
\end{gather*}

The lifting and flattening maps induce maps between the power sets $\mathcal{P}(A^{nk})$ and $\mathcal{P}((A^k)^n)$ in the natural way:
for any $\varrho \subseteq A^{nk}$ and $\sigma \subseteq (A^k)^n$, define
$\lift{\varrho} := \{\lift{\vect{a}} \mid \vect{a} \in \varrho\}$ and $\flatten{\sigma} := \{\flatten{\vect{a}} \mid \vect{a} \in \sigma\}$.
For relation pairs of suitable arities, we write $\lift{(R,R')} := (\lift{R}, \lift{R'})$ and $\flatten{(R,R')} := (\flatten{R}, \flatten{R'})$.
For a set $Q \subseteq \Relp(A)$ of relation pairs of arbitrary arities, define
\[
\lift{Q} :=
\{ \lift{(R,R')} \mid \text{$(R,R') \in Q$ and the arity of $(R,R')$ is divisible by $k$} \},
\]
and for $T \subseteq \Relp(A^k)$, define
\[
\flatten{T} := \{ \flatten{(R,R')} \mid (R,R') \in T \}.
\]
\end{definition}

\begin{lemma}
\label{lem:flatten}
\leavevmode
\begin{enumerate}[label={\upshape(\roman*)}]
\item\label{lem:flatten:1}
For any $(R,R') \in \Relp(A)$ with arity divisible by $k$ and for any $(T,T') \in \Relp(A^k)$, we have $\flatten{\lift{(R,R')}} = (R,R')$ and $\lift{\flatten{(T,T')}} = (T,T')$.
\item\label{lem:flatten:2}
For any $Q \subseteq \Relp(A)$, $\mcr{Q} = \mcr{\flatten{\lift{\mcr{Q}}}}$.
\item\label{lem:flatten:3}
For any $f \in \Op(A)$ and $(T,T') \in \Relp(A^k)$,
we have $f^{\otimes k} \preserves (T,T')$ if and only if $f \preserves \flatten{(T,T')}$.
\item\label{lem:flatten:4}
For any $\mathcal{M} \subseteq \Op(A)$ and $(T,T') \in \Relp(A^k)$,
$(T,T') \in \mInv \mathcal{M}^{\otimes k}$ if and only if $\flatten{(T,T')} \in \mInv \mathcal{M}$.
\item\label{lem:flatten:5}
For any $\mathcal{M} \subseteq \Op(A)$,
$\mInv \mathcal{M}^{\otimes k} = \lift{(\mInv \mathcal{M})}$.
\end{enumerate}
\end{lemma}

\begin{proof}
\ref{lem:flatten:1}
Obvious.

\ref{lem:flatten:2}
Let $\varrho \in \mcr{Q}$ be $m$\hyp{}ary.
Then $\underbrace{\varrho \times \dots \times \varrho}_k$ is a $km$\hyp{}ary member of $\mcr{Q}$.
By \ref{lem:flatten:1}, $\varrho \times \dots \times \varrho = \flatten{\lift{(\varrho \times \dots \times \varrho)}} \in \flatten{\lift{\mcr{Q}}}$;
hence $\varrho = \pr_{1,\dots,m}(\varrho \times \dots \times \varrho) \in \mcr{\flatten{\lift{\mcr{Q}}}}$,
so we have $\mcr{Q} \subseteq \mcr{\flatten{\lift{\mcr{Q}}}}$.
The converse inclusion holds because
\[
\flatten{\lift{\mcr{Q}}}
= \{ \varrho \in \mcr{Q} \mid \text{the arity of $\varrho$ is divisible by $k$} \}
\subseteq \mcr{Q};
\]
hence
$\mcr{\flatten{\lift{\mcr{Q}}}} \subseteq \mcr{\mcr{Q}} = \mcr{Q}$.

\ref{lem:flatten:3}
Assume $f$ is $n$\hyp{}ary and $(T,T')$ is $m$\hyp{}ary.
Since $f^{\otimes k}$ is defined as the coordinatewise application of $f$ to $k$\hyp{}tuples in $A^k$, it is easy to see that for any $\vect{a}^i \in (A^k)^m$ ($i = 1, \dots, n$), we have
$\flatten{f^{\otimes k}(\vect{a}^1, \dots, \vect{a}^n)} = f(\flatten{\vect{a}^1}, \dots, \flatten{\vect{a}^n})$.

Assume first that $f^{\otimes k} \preserves (T,T')$.
Let $\vect{b}^1, \dots, \vect{b}^n \in \flatten{T}$; then $\vect{b}^i = \flatten{\vect{a}^i}$ for some $\vect{a}^i \in T$ ($i = 1, \dots, n$).
Since $f^{\otimes k} \preserves (T,T')$, we have $f^{\otimes k}(\vect{a}^1, \dots, \vect{a}^n) \in T'$.
Consequently $f(\flatten{\vect{a}^1}, \dots, \flatten{\vect{a}^n}) = \flatten{f^{\otimes k}(\vect{a}^1, \dots, \vect{a}^n)} \in \flatten{T'}$.
Therefore $f \preserves \flatten{(T,T')}$.

Assume now that $f \preserves \flatten{(T,T')}$.
Let $\vect{a}^1, \dots, \vect{a}^n \in T$.
Then $\flatten{\vect{a}^1}, \dots, \linebreak \flatten{\vect{a}^n} \in \flatten{T}$, so
$\flatten{f^{\otimes k}(\vect{a}^1, \dots, \vect{a}^n)} = f(\flatten{\vect{a}^1}, \dots, \flatten{\vect{a}^n}) \in \flatten{T'}$.
Therefore \linebreak $f^{\otimes k}(\vect{a}^1, \dots, \vect{a}^n) \in T'$; hence $f^{\otimes k} \preserves (T,T')$.

\ref{lem:flatten:4}
Follows immediately from part \ref{lem:flatten:3}.

\ref{lem:flatten:5}
The following logical equivalences hold by parts \ref{lem:flatten:1} and \ref{lem:flatten:4} and by the fact that the arity of $\flatten{(T,T')}$ is a multiple of $k$:
\begin{align*}
(T,T') \in \mInv \mathcal{M}^{\otimes k}
& \iff
\flatten{(T,T')} \in \mInv \mathcal{M}
\\ & \iff
\lift{\flatten{(T,T')}} \in \lift{(\mInv \mathcal{M})}
\\ & \iff
(T,T') \in \lift{(\mInv \mathcal{M})}.
\qedhere
\end{align*}
\end{proof}

\begin{lemma}
\label{lem:flatQ}
For any $Q \subseteq \Relp(A^k)$, $\mcr{\flatten{Q}} = \mcr{\flatten{\mcr{Q}}}$.
\end{lemma}

\begin{proof}
Clearly $\flatten{Q} \subseteq \flatten{\mcr{Q}}$, so the inclusion $\mcr{\flatten{Q}} \subseteq \mcr{\flatten{\mcr{Q}}}$ holds.
For the converse inclusion $\mcr{\flatten{\mcr{Q}}} \subseteq \mcr{\flatten{Q}}$, it suffices to show that $\flatten{\mcr{Q}} \subseteq \mcr{\flatten{Q}}$.
For this, we need to prove the following:
if $(R,R'), (T,T') \in \mcr{Q}$ such that $\flatten{(R,R')}, \flatten{(T,T')} \in \mcr{\flatten{Q}}$, then also
$\flatten{\zeta(R,R')}$, $\flatten{\tau(R,R')}$, $\flatten{\pr(R,R')}$, \linebreak $\flatten{((R,R') \times (T,T'))}$, $\flatten{((R,R') \wedge (T,T'))}$ are in $\mcr{\flatten{Q}}$;
if $(R_i,R'_i) \in \mcr{Q}$ such that $\flatten{(R_i,R'_i)} \in \mcr{\flatten{Q}}$ for $i \in I$, then also
$\flatten{\bigcap_{i \in I}(R_i,R'_i)} \in \mcr{\flatten{Q}}$;
if $(R,R') \in \mcr{Q}$ such that $\flatten{(R,R')} \in \mcr{\flatten{Q}}$ and $(T,T')$ is a relaxation of $(R,R')$, then also $\flatten{(T,T')} \in \mcr{\flatten{Q}}$;
and for any diagonal relation pair $(\delta_\varrho^m, \delta_\varrho^m)$ on $A^k$, we have $\flatten{(\delta_\varrho^m, \delta_\varrho^m)} \in \mcr{\flatten{\mcr{Q}}}$.
Most of this is routine verification, and for illustration we only provide a detailed proof of $\flatten{\tau(R,R')} \in \mcr{\flatten{Q}}$.

So, assume $(R,R') \in \mcr{Q}$ is $m$\hyp{}ary and $\flatten{(R,R')} \in \mcr{\flatten{Q}}$.
We have
\begin{align*}
\flatten{\tau R}
& = \flatten{\{(\vect{a}_2, \vect{a}_1, \vect{a}_3, \dots, \vect{a}_m) \mid (\vect{a}_1, \dots, \vect{a}_m) \in R\}}
\\ & = \{\flatten{(\vect{a}_2, \vect{a}_1, \vect{a}_3, \dots, \vect{a}_m)} \mid (\vect{a}_1, \dots, \vect{a}_m) \in R\}
\\ & = \{\flatten{(\vect{a}_2, \vect{a}_1, \vect{a}_3, \dots, \vect{a}_m)} \mid \flatten{(\vect{a}_1, \dots, \vect{a}_m)} \in \flatten{R} \},
\end{align*}
which is a result of a permutation of the rows of $\flatten{R}$, which can be obtained from $\flatten{R}$ by a suitable application of the operations $\zeta$ and $\tau$.
An analogous statement with the same permutation of rows holds for $\flatten{\tau R'}$.
Consequently, $\flatten{\tau(R,R')} \in \mcr{\flatten{(R,R')}} \subseteq \mcr{\flatten{Q}}$.

As for the other statements, it is straightforward to verify that
\begin{gather*}
\flatten{\zeta(R,R')} = \zeta^k(\flatten{(R,R')}), \\
\flatten{\pr(R,R')} = \pr_{k+1, \dots, mk} \flatten{(R,R')}, \\
\flatten{((R,R') \times (T,T'))} = \flatten{(R,R')} \times \flatten{(T,T')}, \\
\flatten{((R,R') \wedge (T,T'))} = \flatten{(R,R')} \wedge \flatten{(T,T')}, \\
\flatten{\bigcap_{i \in I}(R_i,R'_i)} = \bigcap_{i \in I} \flatten{(R_i,R'_i)};
\end{gather*}
if $(T,T')$ is a relaxation of $(R,R')$ then $\flatten{(T,T')}$ is a relaxation of $\flatten{(R,R')}$;
and $\flatten{(\delta_\varrho^m, \delta_\varrho^m)} = (\delta_{\varrho'}^{km}, \delta_{\varrho'}^{km})$ where $\varrho'$ is the equivalence relation on $\{1, \dots, km\}$ given by $i \mathrel{\varrho'} j$ if and only if $\lceil i / k \rceil \mathrel{\varrho} \lceil j / k \rceil$ and $i \equiv j \pmod{k}$
($\lceil x \rceil$ stands for the least integer greater than or equal to $x$).
Therefore our desired conclusion follows.
\end{proof}


\section{Results}\label{sec:results}

Equipped with the tools introduced in the previous sections, we are now ready to develop our main results.
Note that if $A := (A_s)_{s \in S}$ is an $S$\hyp{}sorted set in which the components $A_s$ are all finite, then
$\mcr{Q_A} = \mInv \mPol Q_A$ 
by Theorem~\ref{thm:mPolmInv-mInvmPol}.
We will build our theory under this finiteness assumption.

\begin{definition}
Let $A$ and $B$ be $S$\hyp{}sorted sets, and let
$\mathcal{M}_1 \subseteq \Op(A)$ and $\mathcal{M}_2 \subseteq \Op(B)$
be minions.
A mapping $\lambda \colon \mathcal{M}_1 \to \mathcal{M}_2$ is a \emph{minion homomorphism} if
for every $f \in \mathcal{M}_1$, $\declaration(f) = \declaration(\lambda f)$, and for every minor $f^u_\sigma$, we have $(\lambda f)^u_\sigma = \lambda(f^u_\sigma)$.
\end{definition}

\begin{definition}
Let $\mathcal{M}_1, \mathcal{M}_2 \subseteq \Op(A)$.
We say that $\mathcal{M}_2$ is an \emph{extension} of $\mathcal{M}_1$ if $\mathcal{M}_1 \subseteq \mathcal{M}_2$.
\end{definition}

\begin{definition}
We define the operators $\sE$, $\sR$, $\sP$, $\sPfin$ as follows.
Let $F \subseteq \Op(A)$ for some $S$\hyp{}sorted set $A$, and
let $\mathcal{F}$ be a collection of sets of operations on some $S$\hyp{}sorted set.
Let $\sP F$ be the set of all direct powers of $F$, and let $\sPfin F$ be the set of all finite direct powers of $F$, i.e.,
$\sPfin F := \{F^{\otimes k} \mid k \in \IN \}$.
Let $\sR \mathcal{F}$ be the set of all reflections of members of $\mathcal{F}$.
Let $\sE \mathcal{F}$ be the set of all extensions of members of $\mathcal{F}$.

We can express the operators $\sR$ and $\sP$ in a more algebraic way as follows.
Given a set $\mathcal{M} \subseteq \Op(A)$, we can view $\mathcal{M}$ as an $S$\hyp{}sorted algebra
$\mathbf{A}_\mathcal{M}$ whose carrier is $A$ and fundamental operations are the members of $\mathcal{M}$, more precisely,
as the algebra
$\mathbf{A}_\mathcal{M} = (A; (f^{\mathbf{A}_\mathcal{M}})_{f \in \mathcal{M}})$ of type $\tau = (S, \mathcal{M}, \declaration_\mathcal{M})$ with $\declaration_\mathcal{M} \colon \mathcal{M} \to W(S) \times S$, $f \mapsto \declaration(f)$, and $f^{\mathbf{A}_\mathcal{M}} = f$ for every $f \in \mathcal{M}$.
On the other hand, given an $S$\hyp{}sorted algebra $\mathbf{A} = (A; (f^\mathbf{A})_{f \in I})$, let us denote by $F_\mathbf{A}$ the set of fundamental operations of $\mathbf{A}$, i.e., $F_\mathbf{A} := \{f^\mathbf{A} \mid f \in I\}$.
Obviously $F_{\mathbf{A}_\mathcal{M}} = \mathcal{M}$ for any $\mathcal{M} \subseteq \Op(A)$, but the algebras $\mathbf{A}_{F_\mathbf{A}}$ and $\mathbf{A}$ are not generally the same.

With the above notation, we have $F' \in \sP F$ if and only if there is some set $I$ such that $F' = F_{(\mathbf{A}_F)^I}$.
We have $F' \in \sR \mathcal{F}$ if and only if $F' = F_{(\mathbf{A}_F)_{(h,h')}}$ for some $F \in \mathcal{F}$ and for some reflection $(h,h')$.
\end{definition}

Recall that a reflection $(h,h')$ of $A$ into $B$ is a pair of maps $h = (h_s)_{s \in S_B}$, $h' = (h'_s)_{s \in S_B}$, $h_s \colon B_s \to A_s$, $h'_s \colon A_s \to B_s$ for all $s \in S_B$, i.e., for all $s \in S$ such that $B_s \neq \emptyset$ (see Definition~\ref{def:reflection}).

\begin{proposition}
\label{prop:E}
Let $A$ and $B$ be $S$\hyp{}sorted sets, and assume that the components $A_s$ and $B_s$ are all finite.
Let $\mathcal{M}_1 := \mPol Q_A$ and $\mathcal{M}_2 := \mPol Q_B$ for $Q_A \subseteq \Relp(A)$ and $Q_B \subseteq \Relp(B)$.
Then the following conditions are equivalent.
\begin{enumerate}[label={\textup{(\Roman*)}}]
\item\label{prop:E:1} $\mathcal{M}_2 \in \sE \mathcal{M}_1$.
\item\label{prop:E:2} $Q_B \subseteq \mcr{Q_A}$.
\end{enumerate}
\end{proposition}

\begin{proof}
The condition $\mathcal{M}_2 \in \sE \mathcal{M}_1$ is equivalent to $\mathcal{M}_1 \subseteq \mathcal{M}_2$.
Since the operators of a Galois connection are order\hyp{}reversing, the latter condition is equivalent to
$Q_B \subseteq \mcr{Q_B} = \mInv \mathcal{M}_2 \subseteq \mInv \mathcal{M}_1 = \mcr{Q_A}$.
\end{proof}

\begin{lemma}
\label{lem:mInvmPolQBcorref}
Let $Q_B \subseteq \Relp(B)$, and let $(h,h')$ be a reflection of $A$ into $B$.
Then $Q_B \subseteq \mInv (\mPol Q_B^{(h,h')})_{(h,h')}$.
\end{lemma}

\begin{proof}
The\hspace{-0.1pt} inclusion\hspace{-0.1pt}
$Q_B^{(h,h')} \subseteq \mInv \mPol Q_B^{(h,h')}$
holds\hspace{-0.1pt} by\hspace{-0.1pt} the\hspace{-0.1pt} basic\hspace{-0.1pt} properties\hspace{-0.1pt} of Galois connections.
Proposition~\ref{prop:refl-mInv}\ref{mInvFhh} implies
$Q_B \subseteq \mInv (\mPol Q_B^{(h,h')})_{(h,h')}$.
\end{proof}

\begin{proposition}
\label{prop:R}
Let $A$ and $B$ be $S$\hyp{}sorted sets, and assume that the components $A_s$ and $B_s$ are all finite.
Let $\mathcal{M}_1 := \mPol Q_A$ and $\mathcal{M}_2 := \mPol Q_B$ for $Q_A \subseteq \Relp(A)$ and $Q_B \subseteq \Relp(B)$.
The following conditions are equivalent.
\begin{enumerate}[label={\textup{(\Roman*)}}]
\item\label{prop:R:1} $\mathcal{M}_2 \in \sR \mathcal{M}_1$.
\item\label{prop:R:2} There exists a reflection $(h,h')$ of $A$ into $B$ such that
\begin{enumerate}[label={\textup{(\roman*)}}]
\item\label{prop:R:2:i} $Q_B^{(h,h')} \subseteq \mcr{Q_A}$ and
\item\label{prop:R:2:ii} $\mcr{Q_A}_{(h,h')} \subseteq \mcr{Q_B}$.
\end{enumerate}
\end{enumerate}
\end{proposition}

\begin{proof}
\ref{prop:R:1}~$\Rightarrow$~\ref{prop:R:2}:
Assume that $\mathcal{M}_2 \in \sR \mathcal{M}_1$.
Then there exists a reflection $(h,h')$ of $A$ into $B$ such that $\mathcal{M}_2 = (\mathcal{M}_1)_{(h,h')}$.
By Theorem~\ref{thm:mPolmInv-mInvmPol}\ref{mInvmPol} and Proposition~\ref{prop:refl-mInv}\ref{mInvFhh}, we have
\[
\mcr{Q_B}
= \mInv \mathcal{M}_2
= \mInv (\mathcal{M}_1)_{(h,h')}
= \{ (T,T') \mid (T,T')^{(h,h')} \in \mInv \mathcal{M}_1 \},
\]
which implies
\[
Q_B^{(h,h')}
\subseteq \mcr{Q_B}^{(h,h')}
\subseteq \mInv \mathcal{M}_1
= \mcr{Q_A}.
\]
Furthermore, by Proposition~\ref{prop:Relax-mInv},
\begin{align*}
\mcr{Q_B}
&= \mInv (\mathcal{M}_1)_{(h,h')}
= \Relax(\mInv \mathcal{M}_1)_{(h,h')}
\\
&= \Relax(\mcr{Q_A}_{(h,h')})
\supseteq \mcr{Q_A}_{(h,h')}.
\end{align*}

\ref{prop:R:2}~$\Rightarrow$~\ref{prop:R:1}:
Assume that there exist a reflection $(h,h')$ of $A$ into $B$ such that
$Q_B^{(h,h')} \subseteq \mcr{Q_A}$ and $\mcr{Q_A}_{(h,h')} \subseteq \mcr{Q_B}$.
Then
\begin{align*}
\mathcal{M}_2
& = \mPol Q_B
= \mPol \mcr{Q_B}
\subseteq \mPol \mcr{Q_A}_{(h,h')}
= \mPol \Relax \mcr{Q_A}_{(h,h')}
\\ & = \mPol \mInv (\mathcal{M}_1)_{(h,h')}
= (\mathcal{M}_1)_{(h,h')},
\end{align*}
where we have applied
Proposition~\ref{prop:refl-mc},
Theorem~\ref{thm:mPolmInv-mInvmPol},
Proposition~\ref{prop:Relax-mInv},
Lemma \ref{lem:mPolRelaxQ},
and basic properties of Galois connections together with the inclusion $\mcr{Q_A}_{(h,h')} \subseteq \mcr{Q_B}$.

Conversely, from the inclusion $Q_B^{(h,h')} \subseteq \mcr{Q_A}$ we get
\[
\mPol Q_B^{(h,h')}
\supseteq \mPol \mcr{Q_A}
= \mPol Q_A
= \mathcal{M}_1,
\]
which implies
\[
(\mPol Q_B^{(h,h')})_{(h,h')}
\supseteq
(\mathcal{M}_1)_{(h,h')}.
\]
Thus by Lemma~\ref{lem:mInvmPolQBcorref} and basic properties of Galois connections,
\[
Q_B
\subseteq
\mInv (\mPol Q_B^{(h,h')})_{(h,h')}
\subseteq
\mInv (\mathcal{M}_1)_{(h,h')},
\]
and, consequently,
\[
\mathcal{M}_2
= \mPol Q_B
\supseteq \mPol \mInv (\mathcal{M}_1)_{(h,h')}
= (\mathcal{M}_1)_{(h,h')}.
\]
We conclude that $\mathcal{M}_2 = (\mathcal{M}_1)_{(h,h')}$; hence $\mathcal{M}_2 \in \sR \mathcal{M}_1$.
\end{proof}

\begin{proposition}
\label{prop:P}
Let $A$ and $B$ be $S$\hyp{}sorted sets, and assume that the components $A_s$ and $B_s$ are all finite.
Let $\mathcal{M}_1 := \mPol Q_A$ and $\mathcal{M}_2 := \mPol Q_B$ for $Q_A \subseteq \Relp(A)$ and $Q_B \subseteq \Relp(B)$.
Then the following conditions are equivalent.
\begin{enumerate}[label={\textup{(\Roman*)}}]
\item\label{prop:P:1} $\mathcal{M}_2 \in \sPfin \mathcal{M}_1$.
\item\label{prop:P:2} There exists an integer $k \in \IN_{+}$ such that $B = A^k$ and $\lift{\mcr{Q_A}} = \mcr{Q_B}$.
\end{enumerate}
Moreover, these conditions imply the following:
\begin{enumerate}[label={\textup{(\Roman*)}}, resume]
\item\label{prop:P:3} There exists an integer $k \in \IN_{+}$ such that $B = A^k$ and $\mcr{\flatten{Q_B}} = \mcr{Q_A}$.
\end{enumerate}
\end{proposition}

\begin{proof}
\ref{prop:P:1}~$\Rightarrow$~\ref{prop:P:2}:
Assume $\mathcal{M}_2 \in \sPfin \mathcal{M}_1$.
Then there exists an integer $k \in \IN_{+}$ such that $\mathcal{M}_2 = \mathcal{M}_1^{\otimes k}$; hence $B = A^k$.
By Lemma~\ref{lem:flatten}\ref{lem:flatten:5} we have
\begin{align*}
\mcr{Q_B}
& = \mInv \mPol Q_B
= \mInv \mathcal{M}_2
= \mInv \mathcal{M}_1^{\otimes k}
\\ &
= \lift{(\mInv \mathcal{M}_1)}
= \lift{(\mInv \mPol Q_A)}
= \lift{\mcr{Q_A}}.
\end{align*}

\ref{prop:P:2}~$\Rightarrow$~\ref{prop:P:1}:
Assume there exists an integer $k \in \IN_{+}$ such that $B = A^k$ and $\lift{\mcr{Q_A}} = \mcr{Q_B}$.
By Lemma~\ref{lem:flatten}\ref{lem:flatten:5}, we have
\begin{align*}
\mInv \mathcal{M}_2
& = \mInv \mPol Q_B
= \mcr{Q_B}
= \lift{\mcr{Q_A}}
\\ & = \lift{(\mInv \mPol Q_A)}
= \lift{(\mInv \mathcal{M}_1)}
= \mInv \mathcal{M}_1^{\otimes k}.
\end{align*}
Consequently, $\mathcal{M}_2 = \mPol \mInv \mathcal{M}_2 = \mPol \mInv \mathcal{M}_1^{\otimes k} = \mathcal{M}_1^{\otimes k}$, i.e., $\mathcal{M}_2 \in \sPfin \mathcal{M}_1$.

\ref{prop:P:2}~$\Rightarrow$~\ref{prop:P:3}:
By Lemmas~\ref{lem:flatten}\ref{lem:flatten:2} and \ref{lem:flatQ} we have
$\mcr{Q_A}
= \mcr{\flatten{\lift{\mcr{Q_A}}}}
= \mcr{\flatten{\mcr{Q_B}}}
= \mcr{\flatten{Q_B}}$.
\end{proof}

\begin{lemma}
\label{lem:Qhh}
For any set $Q \subseteq \Relp(A)$ and for any reflection $(h,h')$ of $A$ into $B$ it holds that $(Q_{(h,h')})^{(h,h')} \subseteq \Relax Q \subseteq \mcr{Q}$.
\end{lemma}

\begin{proof}
An element $(T,T') \in (Q_{(h,h')})^{(h,h')}$ is of the form $((R,R')_{(h,h')})^{(h,h')}$ for some $(R,R') \in Q$.
Since $T = h(h^{-1}(R)) \subseteq R$ and $T' = h'^{-1}(h'(R')) \supseteq R'$,
$(T,T')$ is a relaxation of $(R,R')$, so
$(T,T') \in \Relax Q \subseteq \mcr{Q}$.
\end{proof}

\begin{proposition}
\label{prop:ER}
Let $A$ and $B$ be $S$\hyp{}sorted sets, and assume that the components $A_s$ and $B_s$ are all finite.
Let $\mathcal{M}_1 := \mPol Q_A$ and $\mathcal{M}_2 := \mPol Q_B$ for $Q_A \subseteq \Relp(A)$ and $Q_B \subseteq \Relp(B)$.
The following conditions are equivalent.
\begin{enumerate}[label={\textup{(\Roman*)}}]
\item\label{prop:ER:1} $\mathcal{M}_2 \in \mathop{{\sE}{\sR}} \mathcal{M}_1$.
\item\label{prop:ER:2} There exists a reflection $(h,h')$ of $A$ into $B$ such that $Q_B^{(h,h')} \subseteq \mcr{Q_A}$.
\end{enumerate}
\end{proposition}

\begin{proof}
\ref{prop:ER:1}~$\Rightarrow$~\ref{prop:ER:2}:
Assume $\mathcal{M}_2 \in \mathop{{\sE}{\sR}} \mathcal{M}_1$.
Then there exists a minion $\mathcal{M}'_2$ such that $\mathcal{M}'_2 \subseteq \mathcal{M}_2$ and $\mathcal{M}'_2 \in \sR \mathcal{M}_1$.
Let $Q_{B'} := \mInv \mathcal{M}'_2$; obviously $\mathcal{M}'_2 = \mPol Q_{B'}$.
By Proposition~\ref{prop:R}, there exists a reflection $(h,h')$ of $A$ into $B$ such that $Q_{B'}^{(h,h')} \subseteq \mcr{Q_A}$.
Since $\mathcal{M}'_2 \subseteq \mathcal{M}_2$, we have $Q_B \subseteq \mInv \mathcal{M}_2 \subseteq \mInv \mathcal{M}'_2 = Q_{B'}$.
By taking $(h,h')$\hyp{}coreflections, we obtain $Q_B^{(h,h')} \subseteq Q_{B'}^{(h,h')} \subseteq \mcr{Q_A}$.

\ref{prop:ER:2}~$\Rightarrow$~\ref{prop:ER:1}:
Assume there exists a reflection $(h,h')$ of $A$ into $B$ such that $Q_B^{(h,h')} \subseteq \mcr{Q_A}$.
Let $Q_{B'} := Q_B \cup \mcr{Q_A}_{(h,h')}$ and $\mathcal{M}'_2 := \mPol Q_{B'}$.
Since $Q_B \subseteq Q_{B'}$ by definition, we have $\mathcal{M}'_2 = \mPol Q_{B'} \subseteq \mPol Q_B = \mathcal{M}_2$, so $\mathcal{M}_2 \in \sE \mathcal{M}'_2$.
It remains to show that $\mathcal{M}'_2 \in \sR \mathcal{M}_1$.
For this, it suffices to show that $Q_A$ and $Q_{B'}$, together with the reflection $(h,h')$, satisfy the conditions of Proposition~\ref{prop:R}\ref{prop:R:2}.
Condition~\ref{prop:R:2:i} holds because $Q_{B'}^{(h,h')} = Q_B^{(h,h')} \cup (\mcr{Q_A}_{(h,h')})^{(h,h')} \subseteq \mcr{Q_A} \cup \mcr{Q_A} = \mcr{Q_A}$ by Lemma~\ref{lem:Qhh}.
Condition~\ref{prop:R:2:ii} holds because $\mcr{Q_A}_{(h,h')} \subseteq Q_{B'}$ by definition, so $\mcr{Q_A}_{(h,h')} \subseteq \mcr{Q_{B'}}$.
\end{proof}

\begin{remark}
From this point on, we have to make the small technical assumption that the $S$\hyp{}sorted sets $A$ and $B$ satisfy $S_B \subseteq S_A$.
This is due to the fact that reflections of $A$ into $B$ exist if and only if $S_B \subseteq S_A$ (see Definition~\ref{def:reflection}) but minion homomorphisms may exist between minions on $A$ and $B$ regardless of the essential sorts.
For example, consider
$S = \{1\}$, $A = \{0, 1, 2\}$, $B = \emptyset$; hence $S_A = \{1\} = S$, $S_B = \emptyset$.
Let $\mathcal{M}_1 := \{c_2^{(n)} \mid n \in \IN_{+}\} \subseteq \Op(A)$ (constant operations of all arities taking value $2$),
$\mathcal{M}_2 := \{\emptyset^{(n)} \mid n \in \IN_{+}\} \subseteq \Op(B)$ (empty functions of all arities).
The sets $\mathcal{M}_1$ and $\mathcal{M}_2$ are minions, and the maps
$\lambda \colon \mathcal{M}_1 \to \mathcal{M}_2$, $c_2^{(n)} \mapsto \emptyset^{(n)}$, and
$\mu \colon \mathcal{M}_2 \to \mathcal{M}_1$, $\emptyset^{(n)} \mapsto c_2^{(n)}$,
are minion homomorphisms.
\end{remark}

\begin{lemma}
\label{lem:mh-ModmId}
Let $A$ and $B$ be $S$\hyp{}sorted sets such that $S_B \subseteq S_A$.
Let $\mathcal{M}_1 \subseteq \Op(A)$ and $\mathcal{M}_2 \subseteq \Op(B)$ be minions,
There exists a surjective minion homomorphism of $\mathcal{M}_1$ onto $\mathcal{M}_2$ if and only if there exists an algebra $\mathbf{B} \in \mathop{{\sR}{\sP}} \mathbf{A}_{\mathcal{M}_1}$ with $F_\mathbf{B} = \mathcal{M}_2$.
\end{lemma}

\begin{proof}
For notational simplicity, write $\mathbf{A} := \mathbf{A}_{\mathcal{M}_1}$.
Assume first that $\lambda \colon \mathcal{M}_1 \to \mathcal{M}_2$ is a surjective minion homomorphism.
Let $\mathbf{B} = (B, (f^\mathbf{B})_{f \in \mathcal{M}_1})$ be the algebra of the same type as $\mathbf{A}_{\mathcal{M}_1}$ with fundamental operations $f^\mathbf{B} = \lambda f$ for every $f \in \mathcal{M}_1$.
By definition, $F_\mathbf{B} = \mathcal{M}_2$.
It remains to show that $\mathbf{B} \in \mathop{{\sR}{\sP}} \mathbf{A}$, which is equivalent to $\mathbf{B} \in \Mod \mId \mathbf{A}$ by Theorem~\ref{thm:Mod-mId-K}.
Let $f_\sigma \approx_{S'} g_\pi$ be a minor identity satisfied by $\mathbf{A}$.
This means that for any $u \in W(S)$ with $\range u = S'$ that is a feasible arity for both $f_\sigma$ and $g_\pi$, we have $(f^\mathbf{A})_\sigma^u = (g^\mathbf{A})_\pi^u$.
Since $\lambda$ is a minion homomorphism, $(f^\mathbf{B})_\sigma^u = (\lambda f^\mathbf{A})_\sigma^u = \lambda ((f^\mathbf{A})_\sigma^u) = \lambda ((g^\mathbf{A})_\pi^u) = (\lambda g^\mathbf{A})_\pi^u = (g^\mathbf{B})_\pi^u$.
Hence $\mathbf{B}$ satisfies $f_\sigma \approx_{S'} g_\pi$.
We conclude that $\mathbf{B}$ satisfies every identity satisfied by $\mathbf{A}$, that is, $\mathbf{B} \in \Mod \mId \mathbf{A}$.

Assume now that there exists an algebra $\mathbf{B} \in \mathop{{\sR}{\sP}} \mathbf{A} = \Mod \mId \mathbf{A}$ such that $F_\mathbf{B} = \mathcal{M}_2$.
Define the map $\lambda \colon \mathcal{M}_1 \to \mathcal{M}_2$ by the rule $f \mapsto f^\mathbf{B}$ for all $f \in \mathcal{M}_1$.
The map $\lambda$ is surjective onto $F_\mathbf{B} = \mathcal{M}_2$ by definition.
We claim that $\lambda$ is a minion homomorphism.
We have $\declaration(f) = \declaration(\lambda f)$ for all $f \in \mathcal{M}_1$ by definition.
Let now $f \in \mathcal{M}_1$, and let $u = u_1 \dots u_m$ and $\sigma$ be such that the minor $f_\sigma^u$ is defined.
Then $f_\sigma^u \in \mathcal{M}_1$ and $f = f^\mathbf{A}$ and $(f^\mathbf{A})_\sigma^u = f_\sigma^u = ((f_\sigma^u)^\mathbf{A})_\iota^u$, where $\iota \colon \nset{m} \to S$, $i \mapsto u_i$, so $\mathbf{A}$ clearly satisfies the identity $f_\sigma \approx_{S'} (f_\sigma^u)_\iota$, where $S'$ is the union of the sets of input sorts of $f$ and $f^u_\sigma$.
By our assumption, also $\mathbf{B}$ satisfies this identity, so
$(f^\mathbf{B})_\sigma^u = ((f_\sigma^u)^\mathbf{B})_\iota^u$.
Hence
$(\lambda f)_\sigma^u
= (\lambda f^\mathbf{A})_\sigma^u
= (f^\mathbf{B})_\sigma^u
= ((f_\sigma^u)^\mathbf{B})_\iota^u
= (f_\sigma^u)^\mathbf{B}
= \lambda ((f_\sigma^u)^\mathbf{A})
= \lambda (f_\sigma^u)$,
and we conclude that $\lambda$ is a surjective minion homomorphism.
\end{proof}

\begin{theorem}
\label{thm:main-RP}
Let $A$ and $B$ be $S$\hyp{}sorted sets such that $S_B \subseteq S_A$.
Let $\mathcal{M}_1 := \mPol Q_A$ and $\mathcal{M}_2 := \mPol Q_B$ for $Q_A \subseteq \Relp(A)$ and $Q_B \subseteq \Relp(B)$.
\begin{enumerate}[label={\textup{(\alph*)}}]
\item\label{main-RP:arb} The following conditions are equivalent.
\begin{enumerate}[label={\textup{(\roman*)}}]
\item\label{main-RP:RP} $\mathcal{M}_2 \in \mathop{{\sR}{\sP}} \mathcal{M}_1$.
\item\label{main-RP:hom} There exists a surjective minion homomorphism $\lambda \colon \mathcal{M}_1 \to \mathcal{M}_2$.
\end{enumerate}
\item\label{main-RP:fin} Assume that the components $A_s$ and $B_s$ of $A$ and $B$ are all finite. Then the following conditions are equivalent.
\begin{enumerate}[label={\textup{(\roman*)}}]
\item\label{main-RP:RPfin} $\mathcal{M}_2 \in \mathop{{\sR}{\sPfin}} \mathcal{M}_1$.
\item\label{main-RP:mcc} There exist an integer $k \in \IN_+$ and a reflection $(h,h')$ of $A^k$ into $B$ such that
\begin{enumerate}[label={\textup{(\arabic*)}}]
\item\label{main-RP:mcc:1} $\flatten{Q_B^{(h,h')}} \subseteq \mcr{Q_A}$ and
\item\label{main-RP:mcc:2} $(\lift{\mcr{Q_A}})_{(h,h')} \subseteq \mcr{Q_B}$.
\end{enumerate}
\end{enumerate}
\end{enumerate}
\end{theorem}

\begin{proof}
\ref{main-RP:arb}
This is Lemma~\ref{lem:mh-ModmId}.

\ref{main-RP:fin}
\ref{main-RP:RPfin}~$\Rightarrow$~\ref{main-RP:mcc}:
Assume $\mathcal{M}_2 \in \mathop{{\sR}{\sPfin}} \mathcal{M}_1$.
Then there exists $\mathcal{M}'_1 \in \sPfin \mathcal{M}_1$ such that $\mathcal{M}_2 \in \sR \mathcal{M}'_1$.
By Proposition~\ref{prop:P}, there exists $k \in \IN_{+}$ such that $\mathcal{M}'_1 = \mPol Q_{A^k}$ for some $Q_{A^k} \subseteq \Relp(A^k)$ satisfying $\lift{\mcr{Q_A}} = \mcr{Q_{A^k}}$;
this together with Lemma~\ref{lem:flatten}\ref{lem:flatten:2} implies $\mcr{Q_A} = \mcr{\flatten{\lift{\mcr{Q_A}}}} = \mcr{\flatten{\mcr{Q_{A^k}}}}$.
By Proposition~\ref{prop:R}, there exists a reflection $(h,h')$ of $A^k$ into $B$ such that $Q_B^{(h,h')} \subseteq \mcr{Q_{A^k}}$ and $\mcr{Q_{A^k}}_{(h,h')} \subseteq \mcr{Q_B}$.
Putting these equalities and inclusions together, we get
\begin{gather*}
\flatten{Q_B^{(h,h')}}
\subseteq \flatten{\mcr{Q_{A^k}}}
\subseteq \mcr{\flatten{\mcr{Q_{A^k}}}}
= \mcr{Q_A},
\\
(\lift{\mcr{Q_A}})_{(h,h')}
= \mcr{Q_{A^k}}_{(h,h')}
\subseteq \mcr{Q_B}.
\end{gather*}

\ref{main-RP:mcc}~$\Rightarrow$~\ref{main-RP:RPfin}:
Assume that there exist an integer $k \in \IN_+$ and a reflection $(h,h')$ of $A^k$ into $B$ such that
$\flatten{Q_B^{(h,h')}} \subseteq \mcr{Q_A}$ and
$(\lift{\mcr{Q_A}})_{(h,h')} \subseteq \mcr{Q_B}$.
Let $Q_{A^k} := \lift{\mcr{Q_A}}$.
By Lemma~\ref{lem:flatten}\ref{lem:flatten:5} we have $Q_{A^k} = \lift{(\mInv \mathcal{M}_1)} = \mInv \mathcal{M}_1^{\otimes k}$, so $Q_{A^k}$ is minor\hyp{}closed by Theorem~\ref{thm:mPolmInv-mInvmPol}, i.e., $\mcr{Q_{A^k}} = Q_{A^k}$.
By Proposition~\ref{prop:power-mc}, $\mathcal{M}_1^{\otimes k}$ is a minion, so $\mathcal{M}_1^{\otimes k} = \mPol \mInv \mathcal{M}_1^{\otimes} = \mPol Q_{A^k}$.
By the above equalitites and inclusions and Lemma~\ref{lem:flatten}\ref{lem:flatten:1}
we have
\begin{gather*}
Q_B^{(h,h')}
= \lift{\flatten{Q_B^{(h,h')}}}
\subseteq \lift{\mcr{Q_A}}
= Q_{A^k}
= \mcr{Q_{A^k}},
\\
\mcr{Q_{A^k}}_{(h,h')}
= (\lift{\mcr{Q_A}})_{(h,h')}
\subseteq \mcr{Q_B}.
\end{gather*}
Now Proposition~\ref{prop:R} yields $\mathcal{M}_2 \in \sR \mathcal{M}_1^{\otimes k} \subseteq \mathop{{\sR}{\sPfin}} \mathcal{M}_1$.
\end{proof}

As announced in the introduction, the following theorem can be considered as the multisorted analogue of the wonderland theorem \cite[Theorem~1.3]{wonderland}.

\begin{theorem}
\label{thm:main-ERP}
Let $A$ and $B$ be $S$\hyp{}sorted sets such that $S_B \subseteq S_A$.
Let $\mathcal{M}_1 := \mPol Q_A$ and $\mathcal{M}_2 := \mPol Q_B$ for $Q_A \subseteq \Relp(A)$ and $Q_B \subseteq \Relp(B)$.
\begin{enumerate}[label={\textup{(\alph*)}}]
\item\label{main-ERP:arb} The following conditions are equivalent.
\begin{enumerate}[label={\textup{(\roman*)}}]
\item\label{main-ERP:ERP} $\mathcal{M}_2 \in \mathop{{\sE}{\sR}{\sP}} \mathcal{M}_1$.
\item\label{main-ERP:hom} There exists a minion homomorphism $\lambda \colon \mathcal{M}_1 \to \mathcal{M}_2$.
\end{enumerate}
\item\label{main-ERP:fin} Assume that the components $A_s$ and $B_s$ of $A$ and $B$ are all finite. Then the following conditions are equivalent.
\begin{enumerate}[label={\textup{(\roman*)}}]
\item\label{main-ERP:ERPfin} $\mathcal{M}_2 \in \mathop{{\sE}{\sR}{\sPfin}} \mathcal{M}_1$.
\item\label{main-ERP:mcc} There exist an integer $k \in \IN_+$ and
a reflection $(h,h')$ of $A^k$ into $B$
such that $\flatten{Q_B^{(h,h')}} \subseteq \mcr{Q_A}$.
\end{enumerate}
\end{enumerate}
\end{theorem}

\begin{proof}
\ref{main-ERP:arb}
Assume $\mathcal{M}_2 \in \mathop{{\sE}{\sR}{\sP}} \mathcal{M}_1$.
Then there exists $\mathcal{M}'_2 \in \mathop{{\sR}{\sP}} \mathcal{M}_1$ such that $\mathcal{M}'_2 \subseteq \mathcal{M}_2$.
By Lemma~\ref{lem:mh-ModmId}, there exists a surjective minion homomorphism $\lambda \colon \mathcal{M}_1 \to \mathcal{M}'_2$.
By extending the codomain of $\lambda$, we get a minion homomorphism from $\mathcal{M}_1$ to $\mathcal{M}_2$.

Assume now that $\lambda \colon \mathcal{M}_1 \to \mathcal{M}_2$ is a minion homomorphism.
Let $\mathcal{M}'_2 := \range \lambda$.
Now $\lambda$ is clearly a surjective minion homomorphism of $\mathcal{M}_1$ into $\mathcal{M}'_2$, so by Lemma~\ref{lem:mh-ModmId}, $\mathcal{M}'_2 \in \mathop{{\sR}{\sP}} \mathcal{M}_1$.
Since $\mathcal{M}'_2 \subseteq \mathcal{M}_2$, we have $\mathcal{M}_2 \in \mathop{{\sE}{\sR}{\sP}} \mathcal{M}_1$.

\ref{main-ERP:fin}
\ref{main-ERP:ERP}~$\Rightarrow$~\ref{main-ERP:mcc}:
Assume $\mathcal{M}_2 \in \mathop{{\sE}{\sR}{\sPfin}} \mathcal{M}_1$.
Then there exists $k \in \IN$ such that $\mathcal{M}_2 \in \mathop{{\sE}{\sR}} \mathcal{M}_1^{\otimes k}$.
By Proposition~\ref{prop:P}, we have $\mathcal{M}_1^{\otimes k} = \mPol Q_{A^k}$ for $Q_{A^k} := \lift{\mcr{Q_A}}$; moreover $\mcr{\flatten{\mcr{Q_{A^k}}}} = \mcr{Q_A}$.
By Proposition~\ref{prop:ER}, there exists a reflection $(h,h')$ of $A^k$ into $B$ such that $Q_B^{(h,h')} \subseteq \mcr{Q_{A^k}}$.
Consequently, $\flatten{(Q_B^{(h,h')})} \subseteq \flatten{\mcr{Q_{A^k}}} \subseteq \mcr{\flatten{\mcr{Q_{A^k}}}} = \mcr{Q_A}$.

\ref{main-ERP:mcc}~$\Rightarrow$~\ref{main-ERP:ERP}:
Assume $k \in \IN_+$ and $(h,h')$ is a reflection of $A^k$ into $B$ satisfying $\flatten{Q_B^{(h,h')}} \subseteq \mcr{Q_A}$.
Let $Q_{A^k} := \lift{\mcr{Q_A}}$; we have $Q_{A^k} = \lift{(\mInv \mathcal{M}_1)} = \mInv \mathcal{M}_1^{\otimes k}$ by Lemma~\ref{lem:flatten}\ref{lem:flatten:5}, so $\mcr{Q_{A^k}} = Q_{A^k}$.
Then $Q_B^{(h,h')} = \lift{(\flatten{Q_B^{(h,h')}})} \subseteq \lift{\mcr{Q_A}} = Q_{A^k} = \mcr{Q_{A^k}}$,
from which it follows by Proposition~\ref{prop:ER} that $\mathcal{M}_2 \in \mathop{{\sE}{\sR}} \mathcal{M}_1^{\otimes k} \subseteq \mathop{{\sE}{\sR}{\sPfin}} \mathcal{M}_1$.
\end{proof}

Under the additional hypothesis that there are only a finite number of sorts, the four conditions of Theorem~\ref{thm:main-RP} become equivalent, and so do those of Theorem~\ref{thm:main-ERP}.
This is a consequence of the following result.

\begin{proposition}
\label{prop:main:finite}
Let $A$ and $B$ be $S$\hyp{}sorted sets.
Assume that $S_{B}$ is finite, all components $A_s$ and
$B_s$ \textup{(}$s \in S_{B}$\textup{)} are finite, and $S_B \subseteq S_A$.
Let $\mathcal{M}_1 \subseteq \Op(A)$ and $\mathcal{M}_2 \subseteq \Op(B)$ be arbitrary sets of operations, not necessarily minions.
Then the following conditions are equivalent.
\begin{enumerate}[label={\textup{(\roman*)}}]
\item\label{prop:main:finite:RPfin} $\mathcal{M}_2 \in \mathop{{\sR}{\sPfin}} \mathcal{M}_1$,
\item\label{prop:main:finite:RP} $\mathcal{M}_2 \in \mathop{{\sR}{\sP}} \mathcal{M}_1$.

\end{enumerate}
\end{proposition}

\begin{proof}
\ref{prop:main:finite:RPfin}~$\Rightarrow$~\ref{prop:main:finite:RP}: Trivial.

\ref{prop:main:finite:RP}~$\Rightarrow$~\ref{prop:main:finite:RPfin}:
Assume that $\mathcal{M}_2 \in \mathop{{\sR}{\sP}} \mathcal{M}_1$ is a reflection of an infinite power of $\mathcal{M}_1$.
Then there exists a reflection $(h,h')$ from $A^{K}$ (for some infinite set $K$) to $B$ such that $\mathcal{M}_2 = (\mathcal{M}_1^{K})_{(h,h')}$.
We are going to construct a finite subset $k \subseteq K$ and a reflection $(\tilde h,\tilde h')$ from $A^k$ to $B$ such that $(f^{\otimes k})_{(\tilde h,\tilde h')} = (f^{\otimes K})_{(h,h')}$ for each $f \in \mathcal{M}_{1}$; therefore $\mathcal{M}_2 = (\mathcal{M}_1^{\otimes K})_{(h,h')} = (\mathcal{M}_1^{\otimes k})_{(\tilde h,\tilde h')}$, which will finish the proof.
We recall the following notation.
For $f \in \mathcal{M}_1$ with $\declaration(f) = (w,s)$, $w = s_1 \dots s_n$, a multisorted map
$\bar h = (\bar h_{s})_{s \in S} \colon B \to C$, and
$\bb := (b_1, \dots, b_n) \in B_w = B_{s_1} \times \dots \times B_{s_n}$,
let $\bar h_{w}(\bb) := (\bar h_{s_1}(b_1), \dots, \bar h_{s_n}(b_n))$. 
For $\alpha = (a_j)_{j \in K} \in A_s^K$ let $\pr_j(\alpha) := a_j$, $j \in K$, and let $\pr_k(\alpha) := (a_j)_{j \in k}$ be the projection (restriction) onto the coordinates in $k$. 

Now we define $\tilde h := {\pr_k} \circ h$, i.e., $\tilde h_s(b) := \pr_k(h_s(b))$ for $b \in B_s$, $s \in S_B$.
The subset $k$ will be chosen below in such a way that it satisfies
\begin{align}\label{eq:0}
  f^{\otimes k}(\tilde h_w(\bb)) = g^{\otimes k}(\tilde h_v(\bc))
  \implies
  f^{\otimes K}(h_w(\bb)) = g^{\otimes K}(h_v(\bc))
\end{align}
for all $f, g \in \mathcal{M}_1$ with
$\declaration(f) = (w,s)$,
$\declaration(g) = (v,s)$
and all $\bb \in B_w$, $\bc \in B_v$.
Note that $f^{\otimes k}(\tilde h_w(\bb)) = \pr_k(f^{\otimes K}(h_w(\bb)))$.

We define the multisorted map $\tilde h' = (\tilde h'_s)_{s \in S_B} \colon A^k \to B$ as follows:
\begin{align} \label{eq:1}
  \tilde h'_s(\xi) &:=
       \begin{cases}
          h'_s(f^{\otimes K}(h_w(\bb)))
             & \text{if $\xi = f^{\otimes k}(\tilde h_w(\bb))$ for some $f \in \mathcal{M}_1$} \\
             & \text{with $\declaration(f) = (w,s)$ and $\bb \in B_w$,} \\
          \gamma_s
             & \text{otherwise,}
        \end{cases}
\end{align}
where $\gamma_s$ is an arbitrary fixed element of $B_s$, $s \in S_B$.

Because of the property \eqref{eq:0}, $\tilde h'$ is well-defined.
We therefore have
\[
(f^{\otimes k})_{(\tilde h,\tilde h')}(\bb)
= \tilde h'_s(f^{\otimes k}(\tilde h_w(\bb)))
\stackrel{\eqref{eq:1}}{=}
h'_{s}(f^{\otimes K}(h_w(\bb)))
= (f^{\otimes K})_{(h,h')}(\bb),
\]
i.e.,
$(f^{\otimes k})_{(\tilde h,\tilde h')} = (f^{\otimes K})_{(h,h')}$,
which finishes the proof as mentioned above.

It remains to find a subset $k \subseteq K$ with property \eqref{eq:0}.

Assume without loss of generality that $S_B = \{1, \dots, p\}$ for some $p \in \IN_{+}$ and $B_i = \{d_{i1}, d_{i2}, \dots, d_{i n_i}\}$ for each $i \in S_B$.
So $\bd := (d_{11}, \dots, d_{1 n_1}, \dots, d_{p1}, \dots,\linebreak d_{p n_p})$ is a tuple of length $\ell := \sum_{i=1}^{p} n_{i}$ containing all elements of $B$.
For any $(w,s) = (s_1 \dots s_n, i) \in W(S_B) \times S_B$ and $\bb = (b_1, \dots, b_n) \in B_w = B_{s_1} \times \dots \times B_{s_n}$, define the mapping
$\sigma \bb \colon \nset{n} \to \{(i,j) \mid i \in S_B,\, j \in \nset{n_i}\}$ by the rule $i \mapsto (s_i, j)$ if and only if $b_i = d_{s_i j}$.

Then for any $f \in \mathcal{M}_1$ with $\declaration(f) = (w,s)$ and any $\bar h \colon B \to A^L$
(we need here only $L = k$ or $L = K$, and $\bar h = \tilde h$ or $\bar h = h$, resp.) we have 
\begin{align}
  \label{eq:2}
  f^{\otimes L}(\bar h_{w}(\bb)) =
  f^{\otimes L}_{\sigma \bb}(\bar h_u(\bd))
\qquad \text{for $\bb \in B_w$.}
\end{align}
Here $f^{\otimes L}_{\sigma \bb}$ denotes the minor $(f^{\otimes L})_{\sigma \bb}^u$ where the arity $u$ is $1 \dots 1 \dots p \dots p$ (each sort $i$ appears $n_{i}$ times).
With the definition of a minor (Definition~\ref{def:minor}), \eqref{eq:2} can be checked easily.
Note that $\pr_{j}(h_u(\bd)) \in A_u = A_1 \times \dots \times A_1 \times \dots \times A_p \times \dots \times A_p$ (each $A_i$ appears $n_i$ times), i.e., there exist at most $\card{A_u} = \card{A_1^{n_1} \times \dots \times A_p^{n_p}}$ different $\ell$\hyp{}tuples $\pr_j(h_u(\bd))$, $j \in K$, where, in particular, $\card{A_u}$ is finite since all $A_i$, $B_i$, and $S_B$ are finite.
Therefore one can choose a finite subset $k \subseteq K$ (with $\card{k} \leq \card{A_u}$) such that for all $j \in K$ there exists an $i \in k$ such that $\pr_j(h_u(\bd)) = \pr_i(h_u(\bd))$, the latter being $\pr_i(\tilde h_u(\bd))$.
Hence any $f^{\otimes K}_{\sigma \bb}(h_u(\bd))$ is uniquely determined by its projection onto $k$, i.e., by $f^{\otimes k}_{\sigma \bb}(\tilde h_u(\bd))$.
Consequently, \eqref{eq:0} follows immediately from~\eqref{eq:2}.
\end{proof}

\begin{remark}
The wonderland theorem \cite[Theorem~1.3]{wonderland} concerns whether a relational structure can be obtained from another by special relational constructions; this condition is referred to as pp\hyp{}constructibility (see \cite[Definition~3.4, Corollary~3.10]{wonderland}).
A comparison with our Theorem~\ref{thm:main-ERP} suggests the following generalization of the notion of pp\hyp{}constructibility in the multisorted setting.
For relational structures $Q_A \subseteq \Relp(A)$ and $Q_B \subseteq \Relp(B)$, we say that $Q_A$ \emph{mc\hyp{}constructs} $Q_B$, or that $Q_B$ is \emph{mc\hyp{}constructible} from $Q_A$, if there exist an integer $k \in \IN_+$ and a reflection $(h,h')$ of $A^k$ into $B$ such that $\flatten{Q_B^{(h,h')}} \subseteq \mcr{Q_A}$.
(The ``mc'' in ``mc\hyp{}constructible'' connotes ``minor-closed''.)
Thus, condition~\ref{main-ERP:mcc} in Theorem~\ref{thm:main-ERP}\ref{main-ERP:fin} asserts that $Q_B$ is mc\hyp{}constructible from $Q_A$.
\end{remark}


\section*{Acknowledgments}

This work was partially supported by the Funda\c{c}\~ao para a Ci\^encia e a Tecnologia (Portuguese Foundation for Science and Technology) through the project UIDB/00297/2020 (Centro de Matem\'atica e Aplica\c{c}\~oes) and the project \linebreak PTDC/MAT-PUR/31174/2017.


\end{document}